\documentclass[11pt]{article}

\usepackage[english]{babel}
\usepackage{amsmath,amsfonts,amssymb,amsthm}
\usepackage{latexsym}
\usepackage{makeidx}
\usepackage{color}
\usepackage{vmargin}
\usepackage{enumerate}
\setmarginsrb{20mm}{20mm}{20mm}{20mm}%
            {5mm}{5mm}{5mm}{10mm}

\numberwithin{equation}{section}

\numberwithin{table}{section}

\usepackage[all]{xy}

\newtheorem{em-deff}{Definition}[section]

\newtheorem{lemma}[em-deff]{Lemma}
\newtheorem{theorem}[em-deff]{Theorem}
\newtheorem{corollary}[em-deff]{Corollary}

\newtheorem{proposition}[em-deff]{Proposition}

\newtheorem{em-example}[em-deff]{Example}

\newtheorem{em-remark}[em-deff]{Remark}
\newtheorem{question}[em-deff]{Question}

\newenvironment{example}{\begin{em-example} \em }{ \end{em-example}}

\newenvironment{deff}{\begin{em-deff} \em }{ \end{em-deff}}

\newcommand{\bigslant}[2]{{\raisebox{.3em}{$#1$}\left/\raisebox{-.3em}{$#2$}\right.}}

\newcommand{\FF}{\mathbb F}
\newcommand{\K}{\mathbb K}
\newcommand{\N}{\mathbb N}
\newcommand{\Z}{\mathbb Z}

\newcommand{\B}{\mathcal{B}}
\newcommand{\BV}{\B(V)}

\newcommand{\F}{\mathbb{F}}

\newcommand{\ca}[1]{\mathcal{#1}}

\def\f{\phi}

\def\ent{\mathrm{ent}}

\def\End{\mathrm{End}}

\def\End{\mathrm{End}}

\def\Im{\mathrm{Im}}

\title{Algebraic entropy in locally linearly compact vector spaces}

\author{Ilaria Castellano \and Anna Giordano Bruno}

\date{\emph{Dedicated to the 70eth birthday of Luigi Salce}}

\begin{document}

\maketitle

\abstract{We introduce the algebraic entropy for continuous endomorphisms of locally linearly compact vector spaces over a discrete field, as the natural extension of the algebraic entropy for endomorphisms of discrete vector spaces studied in \cite{GBSalce}. We show that the main properties of entropy continue to hold in the general context of locally linearly compact vector spaces, in particular we extend the Addition Theorem.}

%
%


\section{Introduction}\label{intro}

In \cite{AKM} Adler, Konheim and McAndrew introduced  the notion of topological entropy $h_{top}$ for continuous self-maps of compact spaces, and they concluded the paper by sketching a definition of the algebraic entropy $h_{alg}$ for endomorphisms of abelian groups. This notion of algebraic entropy, which is appropriate for torsion abelian groups and vanishes on torsion-free abelian groups, was later reconsidered by Weiss in \cite{Weiss}, who proved all the basic properties of $h_{alg}$. Recently, $h_{alg}$ was deeply investigated by Dikranjan, Goldsmith, Salce and Zanardo for torsion abelian groups in \cite{Aeag}, where they proved in particular the Addition Theorem and the Uniqueness Theorem.

Later on, Peters suggested another definition of algebraic entropy for automorphisms of abelian groups in \cite{Peters1}; here we denote Peters' entropy still by $h_{alg}$, since it coincides with Weiss' notion on torsion abelian groups; on the other hand, Peters' entropy is not vanishing on torsion-free abelian groups. In \cite{DGBpet} $h_{alg}$ was extended to all endomorphisms and deeply investigated, in particular the Addition Theorem and the Uniqueness Theorem were proved in full generality.
In \cite{Peters2} Peters gave a further generalization of his notion of entropy for continuous automorphisms of locally compact abelian groups, which was recently extended by Virili in  \cite{Virili} to continuous endomorphisms.

Weiss in \cite{Weiss} connected the algebraic entropy $h_{alg}$ for endomorphisms of torsion abelian groups with the topological entropy $h_{top}$ for continuous endomorphisms of totally disconnected compact abelian groups by means of Pontryagin duality. Moreover, the same connection was shown by Peters in \cite{Peters1} between $h_{alg}$ for topological automorphisms of countable abelian groups and $h_{top}$ for topological automorphisms of metrizable compact abelian groups. These results, known as Bridge Theorems, were recently extended to endomorphisms of abelian groups in \cite{DGB-BT}, to continuous endomorphisms of locally compact abelian groups with totally disconnected Pontryagin dual in \cite{GBD}, and to topological automorphisms of locally compact abelian groups in \cite{Virili-BT} (in the latter two cases on the Potryagin dual one considers an extension of $h_{top}$ to locally compact groups based on a notion of entropy introduced by Hood in \cite{H} as a generalization of Bowen's entropy from \cite{B} -- see also \cite{GBVirili}).

\medskip
A generalization of Weiss' entropy in another direction was given in \cite{SZ}, where Salce and Zanardo introduced the $i$-entropy $\ent_i$ for endomorphisms of modules over a ring $R$ and an invariant $i$ of $\mathrm{Mod}(R)$. For abelian groups (i.e., $\Z$-modules) and $i=\log|-|$, $\ent_i$ coincides with Weiss' entropy. Moreover, the theory of the entropies $\ent_L$ where $L$ is a length function was pushed further in \cite{SVV,SV}.

\smallskip
In \cite{GBSalce} the easiest case of $\ent_i$ was studied, namely, the case of vector spaces with the dimension as invariant, as an introduction to algebraic entropy in the most convenient and familiar setting.
The \emph{dimension entropy} $\ent_{\dim}$ is defined for an endomorphism $\phi:V\to V$ of a vector space $V$ as
$$\ent_{\dim }(\f)=\sup\{H_{\dim}(\f,F): F\leq V,\ \dim F<\infty\},$$
where $$H_{\dim }(\f,F)=\lim_{n\to\infty}\frac{1}{n}\dim(F+\phi F+\ldots+\phi^{n-1}F).$$
All the basic properties of $\ent_{\dim}$ were proved in \cite{GBSalce}, namely, Invariance under conjugation, Monotonicity for linear subspaces and quotient vector spaces, Logarithmic Law, Continuity on direct limits, weak Addition Theorem (see Section~\ref{ss:Properties} for the precise meaning of these properties). 

Moreover, compared to the Addition Theorem for $h_{alg}$ and other entropies, a simpler proof was given in \cite[Theorem 5.1]{GBSalce} of the Addition Theorem for $\ent_{\dim}$, which states that if $V$ is a vector space, $\phi:V\to V$ an endomorphism and $W$ a $\phi$-invariant (i.e., $\phi W\leq W$) linear subspace of $V$, then $$\ent_{\dim}(\phi)=\ent_{\dim}(\phi\restriction_W)+\ent_{\dim}(\overline \phi),$$ where $\overline\phi:V/W\to V/W$ is the endomorphism induced by $\phi$.

Also the Uniqueness Theorem is proved for the dimension entropy (see \cite[Theorem 5.3]{GBSalce}), namely  $\ent_{\dim}$ is the unique collection 
of functions $\ent_{\dim}^V:\End(V)\to\N\cup\{\infty\}$, $\phi\mapsto\ent_{\dim}(\phi)$,
satisfying for every vector space $V$: Invariance under conjugation, Continuity on direct limits, Addition Theorem and $\ent_{\dim}(\beta_F)=\dim F$ for any finite-dimensional vector space $F$, where $\beta_F:\bigoplus_\N F\to \bigoplus_\N F$, $(x_0,x_1,x_2,\ldots)\mapsto (0,x_0,x_1,\ldots)$ is the right Bernoulli shift.  

\medskip
Inspired by the extension of $h_{alg}$ from the discrete case to the locally compact one, and by the approach used in \cite{intrinsic} to define the intrinsic algebraic entropy, we extend the dimension entropy to continuous endomorphisms of locally linearly compact vector spaces. Recall that a linearly topologized vector space $V$ over a discrete field $\K$ is \emph{locally linearly compact} (briefly, l.l.c.) if it admits a local basis at $0$ consisting of linearly compact open linear subspaces; we denote by $\BV$ the family of all linearly compact open linear subspaces of $V$ (see \cite{Lef}).
Clearly, linearly compact and discrete vector spaces are l.l.c.. (See Section~\ref{s:llc} for some background on linearly compact and locally linearly compact vector spaces.)

Let $V$ be an l.l.c.\! vector space and $\f\colon V\to V$ a continuous endomorphism.  The \emph{algebraic entropy of $\phi$ with respect to $U\in\BV$} is
\begin{equation}\label{hdef}
H(\f,U)=\lim_{n\to\infty}\frac{1}{n}\dim\frac{U+\phi U+\ldots+\phi^{n-1}U}{U},
\end{equation}
and the \emph{algebraic entropy} of $\f$ is
\begin{equation*}
\ent(\f)=\sup\{ H(\f,U)\mid U\in\mathcal B(V)\}.
\end{equation*}
In Section \ref{s:ent} we show that the limit in \eqref{hdef} exists. Moreover, we see in Corollary~\ref{lc0} that $\ent$ is always zero on linearly compact vector spaces. On the other hand, if $V$ is a discrete vector space, then $\ent(\f)$ turns out to coincide with $\ent_{\dim}(\phi)$ (see Lemma~\ref{entdim}). Moreover, if $V$ is an l.l.c.\! vector space over a finite field $\mathbb F$, then $V$ is a totally disconnected locally compact abelian group and $h_{alg}(\phi)=\ent(\phi)\cdot\log|\mathbb F|$ (see Lemma~\ref{halg}).

\smallskip
In Section \ref{ss:Properties} we prove all of the general properties that the algebraic entropy is expected to satisfy, namely, Invariance under conjugation, Monotonicity for linear subspaces and quotient vector spaces, Logarithmic Law, Continuity on direct limits, weak Addition Theorem. As a consequence of the computation of the algebraic entropy for the Bernoulli shifts (see Example~\ref{bernoulli}), we find in particular that the algebraic entropy for continuous endomorphisms of l.l.c.\! vector spaces takes all values in $\N\cup\{\infty\}$.

\smallskip
In Section \ref{lff-sec} we prove the so-called Limit-free Formula for the computation of the algebraic entropy, that permits to avoid the limit in the definition in \eqref{hdef} (see Proposition~\ref{lf-ent}). Indeed, taken $V$ an l.l.c.\! vector space and $\phi:V\to V$ a continuous endomorphism, for every $U\in\BV$ we construct an open linear subspace $U^-$ of $V$ (see Definition~\ref{def:uminus}) such that $\phi^{-1}U^-$ is an open linear subspace of $U^-$ of finite codimension and
\begin{equation*}
H(\phi,U)=\dim \frac{U^-}{\phi^{-1}U^-}.
\end{equation*}

A first Limit-free Formula for $h_{alg}$ in the case of injective endomorphisms of torsion abelian groups was sketched by Yuzvinski in \cite{Y} and was later proved in a slightly more general setting in \cite{DGB-lff}; this result was extended in \cite[Lemma 5.4]{yuzapp} to a Limit-free Formula for the intrinsic algebraic entropy of automorphisms of abelian groups.
In \cite{DGB-lff} one can find also a Limit-free Formula for the topological entropy of surjective continuous endomorphisms of totally disconnected compact groups, which was extended to continuous endomorphisms of totally disconnected locally compact groups in \cite[Proposition 3.9]{GBVirili}, using ideas by Willis in \cite{Willis}.
Our Limit-free Formula is inspired by all these results, mainly by ideas from the latter one.

\smallskip
The Limit-free Formula is one of the main tools that we use in Section \ref{AT-sec} to extend the Addition Theorem from the discrete case (i.e., the Addition Theorem for $\ent_{\dim}$ \cite[Theorem 5.1]{GBSalce}) to the general case of l.l.c\! vector spaces (see Theorem \ref{thm:ATllc}). If $V$ is an l.l.c.\! vector space, $\phi:V\to V$ a continuous endomorphism and $W$ a closed $\phi$-invariant linear subspace of $V$, consider the following commutative diagram
\begin{equation*}
\xymatrix{
0\ar[r]&W\ar[r]\ar[d]^{\f\restriction_W}&V\ar[r]\ar[d]^{\f}&V/W\ar[r]\ar[d]^{\overline\f}&0\\
0\ar[r]&W\ar[r]&V\ar[r]&V/W\ar[r]&0}
\end{equation*}
of continuous endomorphisms of l.l.c.\! vector spaces, where $\f\restriction_W$ is the restriction of $\f$ to $W$ and $\overline\f$ is induced by $\phi$; we say that the Addition Theorem holds if $$\ent(\f)=\ent(\f\restriction_W)+\ent(\overline\f).$$

\smallskip
While it is known that $h_{alg}$ satisfies the Addition Theorem for endomorphisms of discrete abelian groups (see \cite{DGBpet}), it is still an open problem to establish whether $h_{alg}$ satisfies the Addition Theorem in the general case of continuous endomorphisms of locally compact abelian groups; from the Addition Theorem for the topological entropy in \cite{GBVirili} and the Bridge Theorem in \cite{GBD} one can only deduce that the Addition Theorem holds for $h_{alg}$ in the case of topological automorphisms of locally compact abelian groups which are compactly covered (i.e., they have totally disconnected Pontryagin dual). 
Here, Theorem~\ref{thm:ATllc} shows in particular that the Addition Theorem holds for $h_{alg}$ on the small subclass of compactly covered locally compact abelian groups consisting of all locally linearly compact spaces over finite fields.

\medskip
With respect to the Uniqueness Theorem for $\ent_{\dim}$ mentioned above, we leave open the following question.

\begin{question}
Does a Uniqueness Theorem hold also for the algebraic entropy $\ent$ on locally linearly compact vector spaces?
\end{question}

In other words, we ask whether $\ent$ is the unique collection of functions $\ent^V:\End(V)\to\N\cup\{\infty\}$, $\phi\mapsto\ent(\phi)$,
satisfying for every l.l.c.\! vector space $V$: Invariance under conjugation, Continuity on direct limits, Addition Theorem and $\ent(\beta_F)=\dim F$ for any finite-dimensional vector space $F$, where $V=\bigoplus_{n=-\infty}^0 F\oplus \prod_{n=1}^\infty F$ is endowed with the topology inherited from the product topology of $\prod_{n\in\Z}F$, and $\beta_F:V \to V$, $(x_n)_{n\in\Z}\mapsto (x_{n-1})_{n\in\Z}$ is the right Bernoulli shift (see Example \ref{bernoulli}).

\medskip
We end by remarking that in \cite{CGB} we introduce a topological entropy for l.l.c.\! vector spaces and connect it to the algebraic entropy studied in this paper by means of Lefschetz Duality, by proving a Bridge Theorem in analogy to the ones recalled above for $h_{alg}$ and $h_{top}$ in the case of locally compact abelian groups and their continuous endomorphisms.

\section{Background on locally linearly compact vector spaces}\label{s:llc}

Fix an arbitrary field $\K$ endowed always with the discrete topology. A topological vector space $V$ over $\K$ is said to be \emph{linearly topologized} if it is Hausdorff and it admits a neighborhood basis at $0$ consisting of linear subspaces of $V$. Clearly, a discrete vector space $V$ is linearly topologized, and if $V$ has finite dimension then the vice-versa holds as well (see \cite[ p.76, (25.6)]{Lef}).

If $W$ is a linear subspace of a linearly topologized vector space $V$, then $W$ with the induced topology is a linearly topologized vector space; if $W$ is also closed in $V$, then $V/W$ with the quotient topology is a linearly topologized vector space as well.

\medskip
Given a linearly topologized vector space $V$, a  \emph{linear variety} $M$ of $V$ is a subset $v+W$, where $v\in V$ and $W$ is a linear subspace of $V$. A linear variety $M=v+W$ is said to be \emph{open} (respectively, \emph{closed}) in $V$ if $W$ is open (respectively, closed) in $V$. 

A  linearly topologized vector space $V$ is \emph{linearly compact} if any collection of closed linear varieties of $V$ with the finite intersection property has non-empty intersection (equivalently, any collection of open linear varieties of $V$ with the finite intersection property has non-empty intersection) (see \cite{Lef}).


\medskip
For reader's convenience, we collect in the following proposition all those properties concerning linearly compact vector spaces that we use further on.

\begin{proposition}\label{prop:lc properties}
Let $V$ be a linearly topologized vector space.
\begin{itemize}
\item[(a)] If $W$ is a linearly compact subspace of $V$, then $W$ is closed.
\item[(b)] If $V$ is linearly compact and $W$ is a closed linear subspace of $V$, then $W$ is linearly compact.
\item[(c)] If $W$ is a linearly topologized vector space and $\phi:V\to W$ is a surjective continuous homomorphism, then $W$ is linearly compact.
\item[(d)] If $V$ is discrete, then $V$ is linearly compact if and only if it has finite dimension (hence, if $V$ has finite dimension then $V$ is linearly compact).
\item[(e)] If $W$ is a closed linear subspace of $V$, then $V$ is linearly compact if and only if $W$ and $V/W$ are linearly compact. 
\item[(f)] The direct product of linearly compact vector spaces is linearly compact.
\item[(g)] An inverse limit of linearly compact vector spaces is linearly compact.
\item[(h)] A linearly compact vector space is complete.
\end{itemize}
\end{proposition}
\begin{proof} 
A proof for (a), (b), (c) and (d) can be found in \cite[page 78]{Lef}. Properties (e) and (f) are proved in \cite[Propositions 2 and 9]{Mat}. Finally, (g) follows from (b) and (f). Let $\iota\colon V\to\tilde V$ be the topological dense embedding of $V$ into its completion $\tilde{V}$, thus (a) implies (h).
\end{proof}

A linearly topologized vector space $V$ is \emph{locally linearly compact} (briefly, l.l.c.\!) if there exists an open linear subspace of $V$ that is linearly compact (see \cite{Lef}). Thus $V$ is l.l.c.\! if and only if it admits a neighborhood basis at $0$ consisting of linearly compact linear subspaces of $V$. Linearly compact and discrete vector spaces are l.l.c.\! vector spaces, of course. 
The structure of an l.l.c.\! vector space can be characterized as follows.

\begin{theorem}[\protect{\cite[(27.10), page 79]{Lef}}] \label{thm:dec}
If $V$ is an l.l.c.\! vector space, then $V$ is topologically isomorphic to $V_c\oplus V_d$, where $V_c$ is a linearly compact linear subspace of $V$ and $V_d$ is a discrete linear subspace of $V$. 
\end{theorem}

By Proposition~\ref{prop:lc properties} and Theorem~\ref{thm:dec}, one may prove that an l.l.c.\! vector space verifies the following properties.

\begin{proposition}\label{rem:complete}
Let $V$ be a linearly topologized vector space.
\begin{itemize}
\item[(a)] If $V$ is l.l.c.\!, then $V$ is complete.
\item[(b)] If $W$ is an l.l.c.\! linear subspace of $V$, then $W$ is closed.
\item[(c)] If $W$ is a closed linear subspace of $V$, then $V$ is l.l.c.\! if and only if $W$ and $V/W$ are l.l.c.\!.
\end{itemize}
\end{proposition}

Given an l.l.c.\! vector space $V$, for the computation of the algebraic entropy we are interested in the neighborhood basis $\BV$ at $0$ of $V$ consisting of all linearly compact open linear subspaces of $V$. 
We see now how the local bases $\B(W)$ and $\B(V/W)$ of a closed linear subspace $W$ of $V$ and the quotient $V/W$ depend on $\BV$.

\begin{proposition}\label{prop:basis}
Let $V$ be an l.l.c.\! vector space and $W$ a closed linear subspace of $V$. Then:
\begin{enumerate}[(a)]
\item $\B(W)=\{U\cap W\mid U\in\BV\}$;
\item $\B(V/W)=\{(U+W)/W\mid U\in\BV\}$.
\end{enumerate}
\end{proposition}
\begin{proof}
(a) Clearly, $\{U\cap W\mid U\in\BV\}\subseteq\B(W)$. Conversely, let $U_W\in\mathcal B(W)$. Since $U_W$ is open in $W$, there exists an open subset $A\subseteq V$ such that $U_W=A\cap W$. As $A$ is a neighborhood of $0$, there exists $U'\in\BV$ such that $U'\subseteq A$. In particular, $U'\cap W\subseteq U_W$ is an open subspace of the linearly compact space $U_W$, and so $U_W/(U'\cap W)$ has finite dimension by Proposition~\ref{prop:lc properties}(d,e). Therefore, there exists a finite-dimensional subspace $F\leq U_W$ such that $U_W=F+(U'\cap W)$. Finally, let $U:=F+U'\in\BV$. Hence, for $F\leq W$ we have $U_W=F+(U'\cap W)=(F+U')\cap W=U\cap W$.

(b) Since the canonical projection $\pi:V\to V/W$ is continuous and open, the set $\{\pi(U)\mid U\in \BV\}$ is contained in $\B(V/W)$.
To prove that $\B(V/W)\subseteq\{(U+W)/W\mid U\in\BV\}$, consider  $\overline U\in\mathcal B(V/W)$ and let $\pi:V\to V/W$ be the canonical projection. Then $\pi^{-1} \overline U$ is an open linear subspace of $V$, hence it contains some $U\in\BV$. Then $\pi U\leq\overline U$ and $\pi U$ has finite codimension in $\overline U$ by Proposition \ref{prop:lc properties}(d,e).
Therefore, there exists a finite-dimensional linear subspace $\overline F$ of $V/W$ such that $\overline F\leq \overline U$ and $\overline U=\pi U+\overline F$. Let $F$ be a finite-dimensional linear subspace of $V$ such that $F\leq \pi^{-1} \overline U$ and $\pi F=\overline F$. Now $\pi(U+F)=\overline U$ and $U+F\in\BV$ by Proposition \ref{prop:lc properties}(c).
\end{proof}

As consequence of Lefschetz Duality Theorem, every linearly compact vector space is topologically isomorphic to a direct product of one-dimensional vector spaces (see \cite[Theorem 32.1]{Lef}). From this result, we derive the known properties that if a linearly topologized vector space $V$ over a finite discrete field is linearly compact then it is compact, and if $V$ is l.l.c.\! then it is locally compact.

\begin{proposition} \label{prop:compact}
Let $V$ be a linearly compact vector space over a discrete field $\K$. Then $V$ is compact if and only if  $\K$ is finite. 
\end{proposition}
\begin{proof} 
Write $V=\prod_{i\in I} \K_i$ with $\K_i=\K$ for all $i\in I$. If $\K$ is finite, then $\K_i$ is compact for all $i\in I$, and so $V$ is compact.
Conversely, if $V$ is compact, then each $\K_i$ is compact as well, hence $\K$ is a compact discrete field, so $\K$ is finite.
\end{proof}

\begin{corollary}\label{cor:tdlc} 
An l.l.c.\! vector space $V$ over a finite discrete field $\FF$ is a totally disconnected locally compact abelian group.
\end{corollary}
\begin{proof} 
By Proposition~\ref{prop:compact}, $\BV$ is a local basis at 0 of $V$ consisting of compact open subgroups, thus van Dantzig Theorem yields the claim.
\end{proof}

\section{Existence of the limit and basic properties}\label{s:ent}

Let $V$ be an l.l.c.\! vector space, $\f:V\to V$ a continuous endomorphism and $U\in\BV$. 
For $n\in\N_+$ and a linear subspace $F$ of $V$, the \emph{$n^{th}$ partial $\phi$-trajectory} of $F$ is
$$T_n(\phi,F)= F+\phi F+\phi^2 F+\ldots +\phi^{n-1} F.$$ 
If $U\in\BV$, notice that for every $n\in\N_+$, $T_n(\phi,U)\in\BV$ as well, as it is open being the union of cosets of $U$, and linearly compact by Proposition \ref{prop:lc properties}(c,f). Moreover, $T_n(\phi,U)\leq T_{n+1}(\phi,U)$ for all $n\in\N_+$, thus we obtain an increasing chain of linearly compact open linear subspaces of $V$, namely
$$U=T_1(\f,U)\leq T_2(\phi,U)\leq\ldots\leq T_{n}(\phi,U)\leq T_{n+1}(\phi,U)\leq\ldots.$$
Moreover, the \emph{$\phi$-trajectory} of $U$ is $T(\phi,U)=\bigcup_{n\in\N_+} T_n (\phi,U),$
which is open and it is the smallest $\f$-invariant linear subspace of $V$ containing $U$. 


Hence, the algebraic entropy of $\phi$ with respect to $U$ introduced in \eqref{hdef} can be written as
\begin{equation}\label{eq:rel ent}
H(\phi,U)=\lim_{n \to \infty} \frac{1}{n}\dim\frac{T_n(\phi,U)}{U}.
\end{equation}
Notice that since $T_n(\f,U)$ is linearly compact and $U$ is open, $U$ has finite codimension in $T_n(\phi,U)$, that is, $\frac{T_n(\phi,U)}{U}$ has finite dimension by Proposition \ref{prop:lc properties}(d,e). Moreover, the following result shows that the limit in \eqref{eq:rel ent} exists.

\begin{proposition}\label{entvalue}
Let $V$ be an l.l.c.\! vector space and $\phi:V\to V$ a continuous endomorphism. For every $n\in\N_+$ let 
\begin{equation*}
\alpha_n=\dim\frac{T_{n+1}(\phi,U)}{T_n(\phi,U)}.
\end{equation*}
Then the sequence of non-negative integers $\{\alpha_n\}_n$ is stationary and $H(\phi,U)=\alpha$ where $\alpha$ is the value of the stationary sequence $\{\alpha_n\}_n$ for $n$ large enough.
\end{proposition}
\begin{proof}
For every $n>1$, $T_{n+1}(\phi,U)=T_n(\phi,U)+\phi^nU$ and  $\phi T_{n-1}(\phi,U)\leq T_{n}(\phi,U)$. Thus 
$$\frac{T_{n+1}(\phi,U)}{T_n(\phi,U)}\cong\frac{\phi^n U}{T_n(\phi,U)\cap\phi^n U}$$ is a quotient of 
$$B_n=\frac{\phi^n U}{\phi T_{n-1}(\phi,U)\cap \phi^n U}.$$ Therefore $\alpha_{n}\leq\dim B_n$. 
Moreover, since $\phi T_n(\phi,U)=\phi T_{n-1}(\phi,U)+\phi^n U$,
$$B_n\cong\frac{\phi T_{n-1}(\phi,U)+\phi^n U}{\phi T_{n-1}(\phi,U)}=\frac{\phi T_n(\phi,U)}{\phi T_{n-1}(\phi,U)}\cong \frac{T_n(\phi,U)}{T_{n-1}(\phi,U)+(T_n(\phi,U)\cap \ker\phi)};$$ 
the latter vector space is a quotient of ${T_n(\phi,U)}/{T_{n-1}(\phi,U)}$, so $\dim B_n\leq \alpha_{n-1}$.
Hence $\alpha_{n}\leq \alpha_{n-1}$. Thus $\{\alpha_n\}_n$ is a decreasing sequence of non-negative integers, therefore stationary.
Since $U\leq T_n(\f,U)\leq T_{n+1}(\f,U)$, 
\begin{equation}\label{eqa1}
\alpha_{n}=\dim\frac{T_{n+1}(\f,U)}{U}-\dim\frac{T_n(\f,U)}{U}.
\end{equation}
As $\{\alpha_n\}_{n}$ is stationary, there exist $n_0>0$ and $\alpha\geq 0$ such that $\alpha_n=\alpha$ for every $n\geq n_0$. 
If $\alpha=0$, equivalently $\dim\frac{T_{n+1}(\f,U)}{U}=\dim\frac{T_n(\f,U)}{U}$ for every $n\geq n_0$, and hence $H(\f,U)=0$. If $\alpha>0$,  by \eqref{eqa1} we have that for every $n\in\N$
$$\dim\frac{T_{n_0+n}(\phi,U)}{U}=\dim\frac{T_{n_0}(\phi,U)}{U}+n\alpha.$$
Thus, 
$$H(\f,U)=\lim_{n\to \infty} \frac{1}{n+n_0} \dim \frac{T_{n_0+n}(\f,U)}{U}=\lim_{n\to\infty}\frac{\dim\frac{T_{n_0}(\f,U)}{U}+n \alpha}{n+n_0}=\alpha.$$
This concludes the proof.
\end{proof}

Proposition \ref{entvalue} yields that the value of $\ent(\phi)$ is either a non-negative integer or $\infty$. Moreover, Example~\ref{bernoulli} below witnesses that $\ent$ takes all values in $\N\cup\{\infty\}$.

\medskip
We see now that the algebraic entropy $\ent$ coincides with $\ent_{\dim}$ on discrete vector spaces.

\begin{lemma}\label{entdim}
Let $V$ be a discrete vector space and $\f\colon V\to V$ an endomorphism. Then $$\ent(\phi)=\ent_{\dim}(\phi).$$
\end{lemma}
\begin{proof}
Note that $\mathcal B(V)=\{F\leq V:\dim F<\infty\}$. Let now $F\in\BV$. Then 
\begin{equation*}\begin{split}
H(\phi,F)=\lim_{n\to\infty}\frac{1}{n}\dim\frac{T_n(\phi,F)}{F}=\lim_{n\to\infty}\frac{1}{n}\left(\dim{T_n(\phi,F)}-\dim{F}\right)=\\=\lim_{n\to\infty}\frac{1}{n}\dim T_n(\phi,F)=H_{\dim}(\phi,F).
\end{split}\end{equation*}
It follows from the definitions that $\ent(\phi)=\ent_{\dim}(\phi)$.
\end{proof}

We compute now the algebraic entropy in the easiest case of the identity automorphism.

\begin{example}\label{ex:id} 
\begin{itemize}
\item[(a)] Let $\f\colon V\to V$ be a continuous endomorphism of an l.l.c.\! vector space $V$.
Then $H(\f,U)=0$ for every $U\in\B(V)$ which is $\phi$-invariant.

\item[(b)] Let $\f=\mathrm{id}_V$. Since every element of $\BV$ is $\f$-invariant, (a) easily implies $\ent(\mathrm{id}_V)=0$.
\end{itemize}
\end{example}

Inspired by the above example we provide now the general case of when the algebraic entropy is zero.

\begin{proposition} 
Let $V$ be an l.l.c.\! vector space, $\f:V\to V$ a continuous endomorphism and $U\in\BV$. Then the following conditions are equivalent:
\begin{enumerate}[(a)]
\item $H(\phi,U)=0$;
\item there exists $n\in\N_+$ such that $T(\phi,U)=T_n(\phi,U)$;  
\item $T(\phi,U)$ is linearly compact.
\end{enumerate}
In particular, $\ent(\phi)=0$ if and only if $T(\phi,U)$ is linearly compact for all $U\in\BV$.
\end{proposition}
\begin{proof} 
(a)$\Rightarrow$(b) If $H(\phi,U)=0$, then $\dim\frac{T_{n+1}(\phi,U)}{T_n(\phi,U)}=0$ eventually by Proposition~\ref{entvalue}. Therefore, the chain of linearly compact open linear subspaces $\left\{ T_n(\phi,U) \right\}_{n\in\N}$ is stationary.

(b)$\Rightarrow$(c) is clear from the definition. 

(c)$\Rightarrow$(a) If $T(\phi,U)$ is linearly compact, by Proposition~\ref{prop:lc properties}(d,e) we have that $\frac{T(\phi,U)}{U}$ is finite-dimensional. Since $T(\phi,U)=\bigcup_{n\in \N_+}T_n(\phi,U)$, it follows that
$$\frac{T(\phi,U)}{U}=\bigcup_{n\in \N_+}\frac{T_n(\phi,U)}{U}$$
and so the chain $\left\{\frac{T_{n}(\phi,U)}{U}\right\}_{n\in\N}$
is stationary. Therefore, $H(\phi,U)=0$.
\end{proof}

As a consequence we see that $\ent$ always vanishes on linearly compact vector spaces.

\begin{corollary} \label{lc0}
If $V$ is a linearly compact vector space and $\phi\colon V\to V$ a continuous endomorphism, then $\ent(\f)=0$. In particular, if $V$ is a finite dimensional vector space, then $\ent(\f)=0$.
\end{corollary}

The next result shows that when $\ent(\phi)$ is finite, this value is realized on some $U\in\BV$.

\begin{lemma}\label{ent=H}
Let $V$ be an l.l.c.\! vector space and $\phi:V\to V$ a continuous endomorphism. If $\ent(\phi)$ is finite, then there exists $U\in\BV$ such that $\ent(\phi)=H(\phi,U)$.
\end{lemma}
\begin{proof}
Since $\ent(\phi)$ is finite and $H(\phi,U)\in\N$ for every $U\in\BV$ by Proposition \ref{entvalue}, the subset $\{H(\phi,U):U\in\BV\}$ of $\N$ is bounded, hence finite. Therefore, $$\ent(\phi)=\sup\{H(\phi,U)\mid U\in\BV\}=\max\{H(\phi,U)\mid U\in\BV\};$$ in other words, $\ent(\phi)=H(\phi,U)$ for some $U\in\BV$ as required.
\end{proof}

We prove now the monotonicity of $H(\f,-)$ on the family $\BV$ ordered by inclusion. 

\begin{lemma}\label{lem:mono}
Let $V$ be an l.l.c.\! vector space and $\f:V\to V$ a continuous endomorphism. If $U,U'\in\BV$ are such that $U'\leq U$, then $H(\f,U')\leq H(\f,U)$.
\end{lemma}
\begin{proof} 
For $n\in\N_+$, since $T_n(\f,U')+U$ is a linear subspace of $T_n(\f,U)$, we have
\begin{equation*}
\frac{T_n(\f,U')/U'}{(T_n(\f,U')\cap U)/U'}\cong \frac{T_n(\f,U')}{T_n(\f,U')\cap U}\cong\frac{T_n(\f,U')+U}{U}\leq \frac{T_n(\f,U)}{U}.
\end{equation*}
Thus,
$$\dim\frac{T_n(\f,U')}{U'}\leq\dim\frac{T_n(\f,U)}{U}+\dim\frac{T_n(\f,U')\cap U}{U'}.$$
Finally, since $\dim\frac{T_n(\f,U')\cap U}{U'}\leq \dim\frac{U}{U'}$, which is constant, for $n\to \infty$ we obtain the thesis.
\end{proof}

Let $(I,\leq)$ be a poset. A subset $J\subseteq I$ is said to be \emph{cofinal} in $I$ if for every $i\in I$ there exists $j\in J$ such that $i\leq j$. The following consequence of Lemma \ref{lem:mono} permits to compute the algebraic entropy on a cofinal subset of $\BV$ ordered by inclusion. 

\begin{corollary}\label{cor:base}
Let $V$ be an l.l.c.\! vector space and $\phi: V \to V$ a continuous endomorphism. 
\begin{enumerate}[(a)]
\item If $\B$ is a cofinal subset of $\BV$, then $\ent(\phi)=\sup\{H(\phi,U)\mid U\in\B\}$.
\item If $U_0\in\BV$ and $\mathcal B=\{U\in\BV:U_0\leq U\}$, then $\ent(\phi)=\sup\{H(\phi,U)\mid U\in\B\}$.
\end{enumerate}
\end{corollary}
\begin{proof}
(a) follows immediately from Lemma~\ref{lem:mono} and the definition.

(b) Since $U_0+U\in\mathcal B$ for every $U\in\BV$, it follows that that $\mathcal B$ is cofinal in $\BV$, so item (a) gives the thesis.
\end{proof}

The following result simplifies the computation of the algebraic entropy in several cases.

\begin{lemma}\label{lem:finite part}
Let $V$ be an l.l.c.\! vector space, $\f\colon V\to V$ a continuous endomorphism and $U\in\BV$. Then there exists a finite-dimensional linear subspace $F$ of $U$ such that, for every $n\in\N_+$, $$T_n(\phi,U)=U+T_n(\phi,F).$$
\end{lemma}
\begin{proof}
We proceed by induction on $n\in\N_+$. For $n=1$ it is obvious. Since $U$ has finite codimension in $T_2(\phi,U)=U+\phi U$, there exists a finite-dimensional linear subspace $F$ of $V$ contained in $U$ and such that $T_2(\phi,U)=U+\phi F=U+T_2(\phi,F)$.
Assume now that $T_n(\phi,U)=U+T_n(\phi,F)$ for some $n\in\N_+$, $n\geq 2$. Then
$$T_{n+1}(\phi,U)=U+\phi T_n(\phi,U)=U+\phi(U)+\phi T_n(\phi,F)=U+\phi F+\phi T_n(\phi,F)=U+T_{n+1}(\phi,F).$$
This concludes the proof.
\end{proof}

We end this section by discussing the relation of $\ent$ with $h_{alg}$.
Recall that a topological abelian group $G$ is \emph{compactly covered} if each element of $G$ is contained in some compact subgroup of $G$ (equivalently, the Pontryagin dual of $G$ is totally disconnected). If $G$ is a compactly covered locally compact abelian group, $\f\colon G\to G$ a continuous endomorphism and $U\in\ca B_{gr}(V)=\{U\leq G\mid\text{$U$ compact open}\}$, then (see~\cite[Theorem 2.3]{GBD})
$$h_{alg}(\f)=\sup\{ H_{alg}(\f,U)\mid U\in\mathcal B_{gr}(V)\}$$
where
$$H_{alg}(\f,U)=\lim_{n\to\infty}\frac{1}{n}\log\left|\frac{T_n(\f,U)}{U}\right|.$$

If $V$ is an l.l.c.\! vector space over a finite field $\F$, by Corollary~\ref{cor:tdlc} it is a totally disconnected locally compact abelian group. In particular $V$ is compactly covered, since $V$ is a torsion abelian group for $\F$ is finite. 

\begin{lemma}\label{halg}
Let $V$ be an l.l.c.\! vector space over a finite field $\F$ and let $\phi:V\to V$ be a continuous endomorphism. Then $$h_{alg}(\f)=\ent(\f)\cdot\log|\F|.$$
\end{lemma}
\begin{proof}
Let $\F=\{f_1,\ldots,f_{|\F|}\}$. Since every $U\in\BV$ is compact by Proposition~\ref{prop:compact}, we have that $U\in\B_{gr}(V)$; hence, $\BV\subseteq \B_{gr}(V)$. 

We show that $\BV$ is cofinal in $B_{gr}(V)$. 
To this end, let $U\in \ca B_{gr}(V)$ and $U'=\sum_{i=1}^{|\F|} f_i U$. Since $V$ is a topological vector space, $f_i U$ is compact for all $i=1,\ldots,|\F|$, so $U'$ is compact as well. Clearly, $U'$ is contained in the linear subspace $\langle U\rangle$ of $V$ generated by $U$. We see that actually $U'=\langle U\rangle$.
Indeed, let 
\begin{equation*}
x=f_{i_1} u_1+\ldots+f_{i_k} u_k,\quad u_1,\ldots, u_k\in U,\quad f_{i_1},\ldots,f_{i_k}\in\F,
\end{equation*}
be an arbitrary element in $\langle U\rangle$. Rearranging the summands, we obtain $x=\sum_{j=1}^{|\F|} f_j  u^j_{l_1\ldots l_j}\in U'$, where for every $j\in\{1,\ldots,|\F|\}$, we let $u^j_{l_1\ldots l_j}=u_{l_1}+\ldots+u_{l_j}\in U$  for $l_1,\ldots,l_j\in\{1,\ldots,k\}$ such that $f_{i_{l_1}}=\ldots=f_{i_{l_j}}=f_j$. Hence $U'=\langle U\rangle$. 
Therefore, $U'\in\BV$ and $U'$ contains $U$,  $\BV$ is cofinal in $\B_{gr}(V)$ as claimed.

Thus, it follows by \cite[Corollary 2.3]{Virili} that $h_{alg}(\f)=\sup\{H_{alg}(\f,U)\mid U\in\BV\}$. Since for every $U\in\BV$, $$\left|\frac{T_n(\f,U)}{U}\right|=|\F|^{\dim\frac{T_n(\f,U)}{U}}$$  for all $n\in\N_+$, we obtain $$H_{alg}(\phi,U)=\lim_{n\to \infty}\frac{1}{n}\log\left|\frac{T_n(\phi,U)}{U}\right|=\lim_{n\to\infty}\frac{1}{n}\left(\dim\frac{T_n(\phi,U)}{U}\log|\F|\right)=H(\phi,U)\log|\F|,$$
and so the thesis follows.
\end{proof}

\section{General properties and examples}\label{ss:Properties}

In this section we prove the general basic properties of the algebraic entropy. These properties extend their counterparts for discrete vector spaces proved for $\ent_{\dim}$ in \cite{GBSalce}. Moreover, our proofs follow those of the same properties for the intrinsic algebraic entropy given  in~\cite{intrinsic}.

\medskip
We start by proving the invariance of $\ent$ under conjugation by a topological isomorphism.


\begin{proposition}[Invariance under conjugation]\label{conj}
Let $V$ be an l.l.c.\! vector space and $\phi:V\to V$ a continuous endomorphism.
If $\alpha\colon V\to W$ is topological isomorphism of l.l.c.\! vector spaces, then $\ent(\phi)=\ent(\alpha\phi\alpha^{-1})$.
\end{proposition}
\begin{proof}
Let $U\in\mathcal{B}(W)$; then $\alpha^{-1}U\in\BV$. For $n\in\N_+$ we have $\alpha T_n(\phi,\alpha^{-1}U)=T_n(\alpha\phi\alpha^{-1},U)$. As $\alpha$ induces an isomorphism $\frac{V}{\alpha^{-1}U}\to \frac{W}{U}$, and furthermore through this isomorphism $\frac{T_n(\phi,\alpha^{-1}U)}{\alpha^{-1}U}$ is isomorphic to $\frac{T_n(\alpha\phi\alpha^{-1},U)}{U}$, by applying the definition we have
$H(\phi,\alpha^{-1}U)=H(\alpha\phi\alpha^{-1},U)$. Now the thesis follows, since $\alpha$ induces a bijection between $\BV$ and $\ca{B}(W)$.
\end{proof}

The next lemma is useful to prove the monotonicity of the algebraic entropy in Proposition \ref{monotonicity}.

\begin{lemma}\label{lem:basis}
Let $V$ be an l.l.c.\! vector space, $\phi:V\to V$ a continuous endomorphism and $W$ a closed $\phi$-invariant linear subspace of $V$. Then:
\begin{eqnarray*}
&&\ent(\f\restriction_W)=\sup\{H(\f\restriction_W,U\cap W)\mid U\in\BV\}\ \text{and}\\
&&\ent(\overline\f)=\sup\{H(\overline\f,(U+W)/W)\mid U\in\BV\},
\end{eqnarray*}
where $\overline\f:V/W\to V/W$ is the continuous endomorphism induced by $\phi$.
\end{lemma} 
\begin{proof} 
Apply Proposition~\ref{prop:basis}.
\end{proof}

Next we see that the algebraic entropy is monotone under taking invariant linear subspaces and quotient vector spaces.

\begin{proposition}[Monotonicity]\label{monotonicity}
Let $V$ be an l.l.c.\!  vector space, $\phi: V \to V$ a continuous endomorphism, $W$ a $\f$-invariant closed linear subspace of $V$ and $\overline\f:V/W\to V/W$ is the continuous endomorphism induced by $\phi$. Then:
\begin{itemize}
\item[(a)] $\ent(\phi)\geq\ent(\phi\restriction_W)$;
\item[(b)] $\ent(\phi)\geq\ent(\overline\f)$.
\end{itemize}
\end{proposition}
\begin{proof}
(a) Let $U\in\BV$ and $n\in\N_+$. Since
$$\frac{T_n(\phi,U)}{U}\geq \frac{U+T_n(\phi\restriction_W,U\cap W)}{U}\cong \frac{T_n(\phi\restriction_W,U\cap W)}{T_n(\phi\restriction_W,U\cap W)\cap U},$$
and $T_n(\phi\restriction_W,U\cap W)\cap U=U\cap W$, it follows that 
\begin{equation*}
\dim\frac{T_n(\phi\restriction_W,U\cap W)}{U\cap W}\leq \dim \frac{T_n(\f,U)}{U}.
\end{equation*}
Hence, $H(\f\restriction_W,U\cap W)\leq H(\f,U)\leq\ent(\f)$. Finally, Lemma~\ref{lem:basis} yields the thesis.  

(b) For $U\in\BV$ and $n\in\N_+$, 
we have that 
\begin{equation}\label{eq:iso1}
\bigslant{T_n\big(\overline\f,\frac{U+W}{W}\big)}{\frac{U+W}{W}}\cong \frac{T_n(\f, U+W)}{U+W}=\frac{T_n(\f,U)+W}{U+W}\cong \frac{T_n(\f,U)}{T_n(\f,U)\cap (U+W)},
\end{equation}
where the latter vector space is clearly a quotient of $\frac{T_n(\phi,U)}{U}$. Therefore,
\begin{equation*}\label{eq:monotonia}
H\left(\overline\f,\frac{U+W}{W}\right)\leq H(\f,U)\leq\ent(\phi).
\end{equation*}
Now Lemma~\ref{lem:basis} concludes the proof.
\end{proof}

Note that equality holds in item (b) of the above proposition if $W$ is also linearly compact. In fact, in this case for every $U\in\BV$ we have $U+W\in\BV$  by Proposition~\ref{prop:lc properties}(c), and hence Lemma~\ref{lem:mono} and the first isomorphism in \eqref{eq:iso1} yield $H(\f,U)\leq H(\f,U+W)=H\left(\overline{\f},\frac{U+W}{W}\right)$; therefore, $\ent(\f)\leq\ent(\overline{\f})$ and so $\ent(\phi)=\ent(\overline\phi)$ by Lemma \ref{monotonicity}(b).

\begin{proposition}[Logarithmic Law]
Let $V$ be an l.l.c.\! vector space and $\phi:V\to V$ a continuous endomorphism. Then $\ent(\phi^k)=k \cdot \ent(\phi)$ for every $k\in\N$.
\end{proposition}
\begin{proof}
For $k=0$, it is enough to note that $\ent(\mathrm{id}_V)=0$ by Example~\ref{ex:id}. So let $k\in\N_+$ and $U\in\BV$. For every $n\in\N_+$, $$T_{nk}(\f,U)=T_n(\f^k,T_k(\f,U))\quad\text{and}\quad T_{n}(\f,T_k(\f,U))=T_{n+k-1}(\f,U).$$
Let $E=T_k(\f,U)\in\BV$. By Lemma~\ref{lem:mono},
\begin{eqnarray*}
k\cdot H(\f,U)&\leq& k\cdot H(\f,E)=k\cdot \lim_{n\to\infty}\frac{1}{nk}\dim \frac{T_{nk}(\f,E)}{E}=\lim_{n\to \infty}\frac{1}{n}\dim \frac{T_{(n+1)k-1}(\f,U)}{E}\\
                                              &\leq& \lim_{n\to \infty}\frac{1}{n}\dim \frac{T_{(n+1)k}(\f,U)}{E} = \lim_{n\to \infty}\frac{1}{n}\dim \frac{T_{n+1}(\f^k,E)}{E} =H(\f^k,E);
\end{eqnarray*}
consequently, $k\cdot \ent(\f)\leq \ent(\f^k)$.

Conversely, as $U\leq E\leq T_{nk}(\f,U)$,
\begin{eqnarray*}
\ent(\f)\geq H(\f,U)&=&\lim_{n\to\infty}\frac{1}{nk}\dim\frac{T_{nk}(\f,U)}{U}=\lim_{n\to\infty}\frac{1}{nk}\dim \frac{T_n(\f^k,E)}{U}\\
&\geq&\lim_{n\to\infty}\frac{1}{nk}\dim\frac{T_n(\f^k,E))}{E}=\frac{1}{k}\cdot H(\f^k,E).
\end{eqnarray*} 
By Lemma~\ref{lem:mono}, it follows that $H(\f^k,E)\geq H(\f^k,U)$, and so $k\cdot \ent(\f)\geq\ent(\f^k)$.
\end{proof}

The next property shows that the algebraic entropy behaves well with respect to direct limits.

\begin{proposition}[Continuity on direct limits]
Let $V$ be an l.l.c.\! vector space and $\phi\colon V\to V$ a continuous endomorphism. Assume that $V$ is the direct limit of a family $\{V_i\mid i\in I\}$ of closed $\phi$-invariant linear subspaces of $V$, and let $\phi_i=\phi\restriction_{V_i}$ for all $i\in I$. Then $\ent(\phi)=\sup_{i\in I}\ent(\phi_i).$
\end{proposition}
\begin{proof}
By Proposition \ref{monotonicity}(a), $\ent(\phi)\geq \ent(\phi_i)$ for every $i\in I$ and so $\ent(\phi)\geq \sup_{i\in I}\ent(\phi_i)$. 

Conversely, let $U\in\BV$. By Lemma~\ref{lem:finite part}, there exists a finite dimensional subspace $F$ of $U$ such that for all $n\in\N_+$
\begin{equation}\label{eqa}
T_n(\phi,U)=U+T_n(\phi,F).
\end{equation}
As $F$ is finite dimensional, $F\leq V_i$ for some $i\in I$. In particular,
\begin{equation}\label{eqb}
T_n(\phi_i,U\cap V_i)=(U\cap V_i)+T_n(\phi,F).
\end{equation} 
Indeed, since $F\leq U\cap V_i$, the inclusion $(U\cap V_i)+T_n(\phi,F)\leq T_n(\phi_i,U\cap V_i)$ follows easily. On the other hand, since $T_n(\phi,F)\leq V_i$,  $$T_n(\phi_i, U\cap V_i)\leq T_n(\phi,U)\cap V_i=(U+T_n(\phi,F))\cap V_i=(U\cap V_i)+T_n(\phi,F).$$
Therefore, \eqref{eqb} yields
$$\frac{T_n(\phi_i,U\cap V_i)}{U\cap V_i}\cong \frac{(U\cap V_i)+T_n(\phi,F)}{U\cap V_i}\cong \frac{T_n(\phi,F)}{U\cap T_n(\phi,F)}.$$
At the same time, \eqref{eqa} implies
$$\frac{T_n(\phi,U)}{U}\cong \frac{U+T_n(\phi,F)}{U}\cong \frac{T_n(\phi,F)}{U\cap T_n(\phi,F)}.$$
Hence, $H(\phi,U)=H(\phi_i, U\cap V_i)\leq \sup_{i\in I}\ent(\phi_i)$, and so $\ent(\phi)\leq \sup_{i\in I}\ent(\phi_i)$.
\end{proof}


We end this list of properties of the algebraic entropy with the following simple case of the Addition Theorem. 

\begin{proposition}[weak Addition Theorem]
For $i=1,2$, let $V_i$ be an l.l.c.\! vector space and $\phi_i:V_1\to V_1$ a continuous endomorphism. Let $V = V_1 \times V_2$ and $\phi= \phi_1 \times \phi_2: V \to V$. Then $\ent(\phi) = \ent(\phi_1) + \ent(\phi_2)$.
\end{proposition}
\begin{proof}
Notice that $\mathcal B=\{U_1\times U_2\mid U_i\in\mathcal{B}(V_i),i=1,2\}$ is cofinal in $\BV$. Indeed, let $U\in\BV$; for $i=1,2$, since the canonical projection $\pi_i\colon V\to V_i$ is an open continuous map, $U_i=\pi_i U\in\mathcal B(V_i)$, and $U\leq U_1\times U_2$. 

Now, for $U_1\times U_2\in\mathcal B$ and for every $n\in\N_+$, 
$$\frac{T_n(\phi,U_1\times U_2)}{U_1\times U_2}\cong \frac{T_n(\phi_1,U_1)}{U_1}\times \frac{T_n(\phi_2,U_2)}{U_2};$$
hence, 
\begin{equation}\label{Heq}
H(\phi,U_1\times U_2)=H(\phi_1,U_1)+H(\phi_2,U_2).
\end{equation}
By Corollary \ref{cor:base}(a) we conclude that $\ent(\phi)\leq \ent(\phi_1)+\ent(\phi_2)$.

If $\ent(\phi)=\infty$, the thesis holds true. So assume that $\ent(\phi)$ is finite; then $\ent(\phi_1)$ and $\ent(\phi_2)$ are finite as well by Proposition \ref{monotonicity}(a).
Hence, for $i=1,2$ by Lemma~\ref{ent=H} there exists $U_i\in\mathcal B(V_i)$ such that $\ent(\phi_i)=H(\phi_i,U_i)$. By \eqref{Heq} we obtain
$$\ent(\phi_1)+\ent(\phi_2)=H(\phi_1,U_1)+H(\phi_2,U_2)=H(\phi,U_1\times U_2)\leq \ent(\phi),$$
where the latter inequality holds because $U_1\times U_2\in\BV$. Therefore, $\ent(\phi_1)+\ent(\phi_2)\leq \ent(\phi)$ and this concludes the proof.
\end{proof}

In the case of a discrete vector space $V$ and an automorphism $\phi:V\to V$, we have that $\ent_{\dim}(\phi^{-1})=\ent_{\dim}(\phi)$ (see \cite{GBSalce}). This property does not extend to the general case of an l.l.c.\! vector space $V$ and a topological automorphism $\phi:V\to V$; in fact, the next example shows that $\ent(\phi)$ could not coincide with $\ent(\phi^{-1})$.

\medskip
Let $F$ be a finite dimensional vector space and let $V=V_c\oplus V_d$, with
$$V_c=\prod_{n=-\infty}^0 F\quad\mbox{and}\quad V_d=\bigoplus_{n=1}^\infty F,$$
be endowed with the topology inherited from the product topology of $\prod_{n\in\Z}F$, so $V_c$ is linearly compact and $V_d$ is discrete.

The \emph{left (two-sided) Bernoulli shift} is
$${}_F\beta\colon V\to V, \quad (x_n)_{n\in\Z}\mapsto (x_{n+1})_{n\in\Z},$$
while the \emph{right (two-sided) Bernoulli shift} is
$$\beta_F\colon V\to V, \quad (x_n)_{n\in\Z}\mapsto(x_{n-1})_{n\in\Z}.$$
Clearly, $\beta_F$ and ${}_F\beta$ are topological automorphisms such that ${}_F\beta^{-1}=\beta_F.$

\smallskip
Let us compute their algebraic entropies.

\begin{example}\label{ex:bs}\label{bernoulli}
\begin{enumerate}[(a)] Consider the case when $F=\mathbb K$, that is, $V_c=\prod_{n=-\infty}^0\K$ and $V_d=\bigoplus_{n=1}^\infty \K.$
%
\item Let $\f\in\{{}_\K\beta,\beta_\K\}$. By Corollary~\ref{cor:base}(b),
\begin{equation*}
\ent(\f)=\sup\{H(\f,U)\mid U\in\BV,\ V_c\leq U\}.
\end{equation*}
Let $U\in\BV$ such that $V_c\leq U$. Since $V_c$ has finite codimension in $U$ by Proposition~\ref{prop:lc properties}(d,e), there exists $k\in\N_+$ such that
$$U\leq U':=\prod_{n=-\infty}^0\K\times\bigoplus_{n=1}^k \K\in\BV,$$
hence $H(\f,U)\leq H(\f,U')$ by Lemma~\ref{lem:mono}. Clearly,
\begin{equation*}
\ldots\leq {}_\K\beta^n(U')\leq\ldots\leq{}_\K\beta(U')\leq U'\leq\beta_\K(U')\leq\ldots\leq\beta_\K^n(U')\leq\ldots.
\end{equation*}
So, for all $n\in\N_+$, $T_n({}_\K\beta,U')=U'$, while
$$\dim\frac{T_{n+1}(\beta_\K,U')}{T_n(\beta_\K,U')}=\dim\frac{\beta^{n+1}_\K(U')}{\beta^n_\K U'}=\dim\frac{\beta_\K(U')}{U'}=1.$$ 
By Corollary~\ref{cor:base}(a), we can conclude that
\begin{equation*}
\ent({}_\K\beta)=0\quad\mbox{and}\quad\ent(\beta_\K)=1.
\end{equation*}
In particular, $\ent(\f)\neq\ent(\f^{-1})$ for $\f\in\{{}_\K\beta,\beta_\K\}$. 

\item It is possible, slightly modifying the computations in item (a), to find that, for $F$ a finite dimensional vector space, $$\ent({}_F\beta)=0\quad \text{and}\quad \ent(\beta_F)=\dim F.$$
\end{enumerate}
\end{example}


\section{Limit-free Formula}\label{lff-sec}

The aim of this subsection is to prove Proposition~\ref{lf-ent}, that provides a formula for the computation of the algebraic entropy avoiding the limit in the definition. 
This formula is a fundamental ingredient in the proof of the Addition Theorem presented in the last section.

\begin{deff}\label{def:uminus}
Let $V$ be an l.l.c.\! vector space, $\phi:V\to V$ a continuous endomorphism and $U\in\BV$. Let:
\begin{enumerate}[-]
\item $U^{(0)}=U$;
\item $U^{(n+1)}=U+\phi^{-1} U^{(n)}$ for every $n\in\N$;
\item $U^-=\bigcup_{n\in\N} U^{(n)}$.
\end{enumerate}
\end{deff}

It can be proved by induction that $U^{(n)}\leq U^{(n+1)}$ for every $n\in\N$.
Since $U$ is open, clearly every $U^{(n)}$ is open as well, so also $U^-$ and $\phi^{-1}U^-$ are open linear subspaces of $V$. 

We see now that $U^-$ is the smallest linear subspace of $V$ containing $U$ and inversely $\phi$-invariant (i.e., $\phi^{-1}U^-\leq U^-$). Note that $U^-$ coincides with $T(\phi^{-1},U)$ when $\phi$ is an automorphism, otherwise it could be strictly smaller.

\begin{lemma}
Let $V$ be an l.l.c.\! vector space, $\phi:V\to V$ a continuous endomorphism and $U\in\BV$. Then:
\begin{enumerate}[(a)]
\item $\phi^{-1}U^-\leq U^-$;
\item if $W$ is a linear subspace of $V$ such that $U\leq W$ and $\phi^{-1}W\leq W$, then $U^-\leq W$.
\end{enumerate}
\end{lemma}
\begin{proof}
(a) follows from the fact that $\phi^{-1}U^{(n)}\leq U^{(n+1)}$ for every $n\in\N$.

(b) By the hypothesis, one can prove by induction that $U^{(n)}\leq W$ for every $n\in\N$; hence, $U^-\leq W$.
\end{proof}

In the next lemma we collect some other properties that we use in the sequel.

\begin{lemma}\label{Umeno}
Let $V$ be an l.l.c.\! vector space, $\phi:V\to V$ a continuous endomorphism and $U\in\BV$. Then:
\begin{enumerate}[(a)]
\item $U^-=U+\phi^{-1}U^-$;
\item $\frac{U^-}{\phi^{-1}U^-}$ has finite dimension.
\end{enumerate}
\end{lemma}
\begin{proof}
(a) follows from the equalities $$U+\phi^{-1}U^-=U+\phi^{-1} \bigcup_{n\in\N}U^{(n)}=U+\bigcup_{n\in\N}\phi^{-1}U^{(n)}=\bigcup_{n\in\N}(U+\phi^{-1}U^{(n)})=U^-.$$

(b) The quotient $\frac{U}{U\cap\phi^{-1}U^-}$ has finite dimension by Proposition~\ref{prop:lc properties}(d,e), since $U\cap\phi^{-1}U^-\leq U$ is open and $U$ is linearly compact. In view of item (a) we have the isomorphism
 $$\frac{U^-}{\phi^{-1}U^-}=\frac{U+\phi^{-1}U^-}{\phi^{-1}U^-}\cong \frac{U}{U\cap\phi^{-1}U^-},$$
so we conclude that also $\frac{U^-}{\phi^{-1}U^-}$ has finite dimension.
\end{proof}

The next lemma is used in the proof of Proposition \ref{lf-ent}.

\begin{lemma}\label{T-U}
Let $V$ be an l.l.c.\! vector space, $\phi:V\to V$ a continuous endomorphism and $U\in\BV$. Then, for every $n\in\N_+$,
$$\phi^{-n}T_{n}(\phi,U)=\phi^{-1}U^{(n-1)}.$$
\end{lemma}
\begin{proof}
We proceed by induction on $n\in\N_+$. We write simply $T_n=T_n(\phi,U)$.

If $n=1$ we have $\phi^{-1}T_1=\phi^{-1}U=\phi^{-1}U^{(0)}$.
Assume now that the property holds for $n\in\N_+$, we prove it for $n+1$, that is, we verify that
\begin{equation}\label{n+1}
\phi^{-(n+1)}T_{n+1}=\phi^{-1}U^{(n)}.
\end{equation}
Let $x\in \phi^{-1}U^{(n)}$. Then,  by inductive hypothesis, $$\phi(x)\in U^{(n)}=U+\phi^{-1}U^{(n-1)}=U+\phi^{-n}T_n.$$
Consequently, $\phi^{n+1}(x)=\phi^n(\phi(x))\in \phi^n U+ T_n=T_{n+1}$; this shows that $x\in\phi^{-(n+1)}T_{n+1}$. Therefore, $\phi^{-1}U^{(n)}\leq \phi^{-(n+1)}T_{n+1}$.

To verify the converse inclusion, let $x\in \phi^{-(n+1)}T_{n+1}$. Then $\phi^{n+1}(x)\in T_{n+1}=T_n+\phi^nU$, so $\phi^{n+1}(x)=y+\phi^n(u)$, for some $y\in T_n$ and $u\in U$.
Therefore, $\phi^n(\phi(x)-u)=y\in T_n$, that is, $\phi(x)-u\in\phi^{-n}T_n=\phi^{-1}U^{(n-1)}$ by inductive hypothesis. Hence, $\phi(x)\in U+\phi^{-1}U^{(n-1)}=U^{(n)}$, and we can conclude that $x\in\phi^{-1}U^{(n)}$.
Thus, \eqref{n+1} is verified. So, the induction principle gives the thesis.
\end{proof}

We are now in position to prove the Limit-free Formula, where clearly we use that $\dim \frac{U^-}{\phi^{-1}U^-}$ has finite dimension by Lemma \ref{Umeno}(b).

\begin{proposition}[Limit-free Formula]\label{lf-ent}
Let $V$ be an l.l.c.\! vector space, $\phi:V\to V$ a continuous endomorphism and $U\in\BV$. Then $$H(\phi,U)=\dim \frac{U^-}{\phi^{-1}U^-}.$$
\end{proposition}
\begin{proof}
We write simply $T_n=T_n(\phi,U)$ for every $n\in\N_+$. By Proposition \ref{entvalue}, there exist $n_0\in\N_+$ and $\alpha\in\N$, such that for every $n\geq n_0$, $H(\phi,U)=\alpha$, where $\alpha=\dim \frac{T_{n+1}}{T_n}$. So, our aim is to prove that $\alpha=\dim \frac{U^-}{\phi^{-1}U^-}$.

For every $n\in\N$, since $U\cap \phi^{-1}U^{(n)}\leq U$ is open and $U$ is linearly compact, by Proposition~\ref{prop:lc properties}(d,e) the quotient $\frac{U}{U\cap \phi^{-1}U^{(n)}}$ has finite dimension; moreover, $U\cap \phi^{-1}U^{(n)}\leq U\cap \phi^{-1}U^{(n+1)}$ so $\frac{U}{U\cap \phi^{-1}U^{(n+1)}}$ is a quotient of $\frac{U}{U\cap \phi^{-1}U^{n}}$. 
The decreasing sequence of finite-dimensional vector spaces $\left\{\frac{U}{U\cap \phi^{-1}U^{n}}\mid n\in\N\right\}$ must stabilize; this means that there exists $n_1\in\N$ such that $U\cap \phi^{-1}U^{(n)}=U\cap \phi^{-1}U^{(n_1)}$ for every $n\geq n_1$. Hence, for every $m\geq n_1$,
$$U\cap \phi^{-1}U^{(m)}=\bigcup_{n\in\N}(U\cap \phi^{-1}U^{(n)})=U\cap \bigcup_{n\in\N}\phi^{-1}U^{(n)}=U\cap\phi^{-1}\bigcup_{n\in\N}U^{(n)}=U\cap \phi^{-1}U^-.$$
Fix now $m\geq\max\{n_0,n_1\}$; since $\frac{U^-}{\phi^{-1}U^-}=\frac{U+\phi^{-1}U}{\phi^{-1}U}\cong\frac{U}{U\cap \phi^{-1}U^-}$ by Lemma \ref{Umeno}(a), we have
\begin{eqnarray}\nonumber
\dim \frac{U^-}{\phi^{-1}U^-}&=&\dim \frac{U}{U\cap \phi^{-1}U^-}=\dim\frac{U}{U\cap \phi^{-1}U^{(m)}}\\ \nonumber
&=&\dim\frac{U+\phi^{-1}U^{(m)}}{\phi^{-1}U^{(m)}}=\dim \frac{U^{(m+1)}}{\phi^{-1}U^{(m)}}.
\end{eqnarray}
We see now that $$\dim\frac{U^{(m)}}{\phi^{-1}U^{(m-1)}}=\dim\frac{T_{m+1}}{T_m}=\alpha$$ and this concludes the proof. To this end, noting that $$\frac{U^{(m)}}{\phi^{-1}U^{(m-1)}}=\frac{U+\phi^{-1}U^{(m-1)}}{\phi^{-1}U^{(m-1)}}\quad\text{and}\quad\frac{T_{m+1}}{T_m}=\frac{\phi^{m+1}U+T_m}{T_m},$$ 
define 
\begin{align*}
\Phi:&\frac{U+\phi^{-1}U^{(m-1)}}{\phi^{-1}U^{(m-1)}}\longrightarrow \frac{\phi^{m+1}U+T_m}{T_m} \\
& x+\phi^{-1}U^{(m-1)}\mapsto\phi^m(x)+T_m.
\end{align*}
Then $\Phi$ is a surjective homomorphism by construction and it is well-defined and injective since $\phi^{-m}T_m=\phi^{-1}U^{(m-1)}$ by Lemma \ref{T-U}.
\end{proof}

\section{Addition Theorem}\label{AT-sec}

%

This section is devoted to the proof of the Addition Theorem for the algebraic entropy $\ent$ for l.l.c.\! vector spaces (see Theorem~\ref{thm:ATllc}).

\medskip
Let $V$ be an l.l.c.\! vector space and $\f:V\to V$ a continuous endomorphism. Theorem~\ref{thm:dec}  allows us to decompose $V$ into the direct sum of a linearly compact open subspace $V_c$ and a discrete linear subspace $V_d$ of $V$, namely, $V\cong V_c\oplus V_d$ topologically. So, assume that $V=V_c\oplus V_d$ and let
\begin{equation}\label{eq:injproj}
\iota_*\colon V_*\to V,\quad p_*\colon V\to V_*,\quad *\in\{c,d\},
\end{equation}
be respectively the canonical embeddings and projections. Accordingly, we may associate to $\f$ the following decomposition
\begin{equation}\label{eq:fdec}
\f=
\begin{pmatrix}
  \f_{cc} & \f_{dc} \\
  \f_{cd} & \f_{dd}
\end{pmatrix},
\end{equation}
where $\phi_{\bullet*}:V_\bullet\to V_*$ is the composition $\phi_{\bullet*}=p_*\circ\phi\circ\iota_\bullet$ for $\bullet, *\in\{c,d\}$.
Therefore, each $\phi_{\bullet*}$ is continuous as it is composition of continuous homomorphisms.

\begin{lemma}\label{lem:kerim} 
In the above notations, consider $\f_{cd}\colon V_c\to V_d$. Then:
\begin{enumerate}[(a)]
\item $\Im(\f_{cd})\in\mathcal B(V_d)$;
\item $\ker(\f_{cd})\in\mathcal B(V_c)\subseteq\BV$.
\end{enumerate}
\end{lemma}
\begin{proof} 
(a) Since $V_d$ is discrete, by Proposition~\ref{prop:lc properties}(c,d) we have that $\Im(\f_{cd})\leq V_d$ has finite dimension, hence $\Im(\f_{cd})\in\mathcal B(V_d)=\{F\leq V_d\mid \dim F<\infty\}$.

(b) As $\ker(\f_{cd})$ is a closed linear subspace of $V_c$, which is linearly compact, then $\ker(\f_{cd})$ is linearly compact as well by Proposition~\ref{prop:lc properties}(b).  Since $V_c/\ker(\f_{cd})\cong\Im(\f_{cd})$ is finite dimensional by item (a), $V_c/\ker(\f_{cd})$ is discrete and so $\ker(\f_{cd})$ is open in $V_c$; therefore, $\ker(\phi_{cd})\in\mathcal B(V_c)$. 
\end{proof}

We show now that the only positive contribution to the algebraic entropy of $\f$ comes from its ``discrete component'' $\phi_{dd}$.

\begin{proposition}\label{prop:resdd}
In the above notations, $\ent(\f)=\ent(\f_{dd})$.
\end{proposition}
\begin{proof} 
By Lemma~\ref{lem:kerim}(a), $\Im(\f_{cd})\in\B(V_d)$; hence, letting $$\B_d=\{F\leq V_d\mid \Im(\f_{cd})\leq F,\dim F<\infty\}\subseteq\B(V_d),$$
Corollary~\ref{cor:base}(b) implies 
\begin{equation}\label{entfdd}
\ent(\f_{dd})=\sup\{H(\f_{dd},F)\mid F\in\B_d\}.
\end{equation}

Let $\mathcal B=\{U\in\BV\mid V_c\leq U\}$, which is cofinal in $\BV$.
For $U\in\B$, since $V_c$ has finite codimension in $U$ by Proposition~\ref{prop:lc properties}(d,e), there exists a finite dimensional linear subspace $F\leq V_d$ such that $U=V_c\oplus F$. Conversely, $V_c\oplus F\in\B$ for every finite dimensional linear subspace $F\leq V_d$.
Hence, $\mathcal B=\{V_c\oplus F\mid  F\in\B(V_d)\}$. Moreover, $\mathcal B'=\{V_c\oplus F\mid F\in\B_d\}$ is cofinal in $\B$ and so in $\BV$.
Thus, Corollary~\ref{cor:base}(b) yields
\begin{equation}\label{entf}
\ent(\f)=\sup\{H(\f,U)\mid U\in\B'\}.
\end{equation}

For $U=V_c\oplus F\in\B'$, as in Definition~\ref{def:uminus} let, for every $n\in\N$,
\begin{eqnarray}
U^{(0)}=U\quad&\mbox{and}&\quad F^{(0)}=F,  \nonumber\\
U^{(n)}=U+\f^{-1}U^{(n-1)}\quad&\mbox{and}&\quad F^{(n)}=F+\f_{dd}^{-1}F^{(n-1)}, \nonumber\\ 
U^-=\bigcup_{n\in\N}U^{(n)}\quad&\mbox{and}&\quad F^-=\bigcup_{n\in\N} F^{(n)}.\nonumber
\end{eqnarray}
Proposition~\ref{lf-ent}, together with \eqref{entf} and \eqref{entfdd} respectively, implies that
\begin{eqnarray}
\ent(\f)&=&\sup\left\{\dim\frac{U^-}{\f^{-1}U^-}\mid U\in\B'\right\},\label{1}\\
\ent(\f_{dd})&=&\sup\left\{\dim\frac{F^-}{\f_{dd}^{-1}F^-}\mid F\in\B_d\right\}.\label{2}
\end{eqnarray}

Let $U=V_c\oplus F\in\B'$. We show by induction on $n\in\N$ that
\begin{equation}\label{eq:EQ}
U^{(n)}= V_c\oplus F^{(n)}\quad \text{for every}\ n\in\N.
\end{equation}
For $n=0$, we have $U^{(0)}=U=V_c\oplus F=V_c\oplus F^{(0)}$. Assume now that $n\in\N$ and that $U^{(n)}=V_c\oplus F^{(n)}$. First note that $U^{(n+1)}=U+\f^{-1}U^{(n)}=U+\f^{-1}(V_c\oplus F^{(n)})$. Moreover, since $\Im(\f_{cd})\leq F\leq F^{(n)}$,
\begin{eqnarray}
\f^{-1}(V_c\oplus F^{(n)})&=&\{(x,y)\in V_c\oplus V_d\mid  \f_{cd}(x)+\f_{dd}(y)\in F^{(n)}\}\nonumber\\
&=&\{(x,y)\in V_c\oplus V_d\mid \f_{dd}(y)\in F^{(n)}\}\nonumber\\
&=&V_c\oplus \f_{dd}^{-1} F^{(n)}.\nonumber
\end{eqnarray}
Thus, $U^{(n+1)}= V_c\oplus F^{(n+1)}$ as required in \eqref{eq:EQ}. 

Now \eqref{eq:EQ} implies that $U^-= V_c\oplus F^-$; moreover, since $\Im(\phi_{cd})\leq F\leq F^-$, $$\f^{-1} U^-=\{(x,y)\in V_c\oplus V_d\mid\phi_{dd}(y)\in F^-\}=V_c\oplus \f_{dd}^{-1}F^-.$$ Therefore, $\frac{U^-}{\phi^{-1}U^-}=\frac{V_c\oplus F^-}{V_c\oplus\phi_{dd}^{-1}F^-}=\frac{F^-}{\phi_{dd}^{-1}F^-}$, so the thesis follows from \eqref{1} and \eqref{2}.
\end{proof}

We are now in position to prove the Addition Theorem.
 
\begin{theorem}[Addition Theorem]\label{thm:ATllc}
Let $V$ be an l.l.c.\! vector space, $\f\colon V\to V$ a continuous endomorphism and $W$ a closed $\phi$-invariant linear subspace of $V$. Then 
$$\ent(\phi)=\ent(\phi\restriction_W)+\ent(\overline\phi),$$
where $\overline\phi:V/W\to V/W$ is the continuous endomorphism induced by $\phi$. 
\end{theorem}
\begin{proof} 
Let $V_c\in\BV$ and $W_c=W\cap V_c\in\B(W)$. By Theorem~\ref{thm:dec}, there exists a discrete linear subspace $W_d$ of $W$ such that $W= W_c\oplus W_d$. Let $V_d$ be a linear subspace of $V$ such that $V=V_c\oplus V_d$ and $W_d\leq V_d$. Clearly, $V_d$ is discrete, since $V_c$ is open and $V_c\cap V_d=0$. By construction, the diagram
\begin{equation*}
\xymatrix{0\ar[r]&W_d\ar@/^2pc/[rrr]^{(\phi\restriction_{W})_{dd}} \ar[r]^{\iota^W_d}\ar@{^{(}->}[d]&W\ar[r]^{\f\restriction_W}&W\ar[r]^{p^W_d}&W_d\ar@{^{(}->}[d]\ar[r]&0\\
0\ar[r]&V_d\ar@/_2pc/[rrr]_{\phi_{dd}} \ar[r]^{\iota^V_d}&V\ar[r]^{\f}&V\ar[r]^{p^V_d}&V_d\ar[r]&0}
\end{equation*}
commutes, where $\iota^W_d, \iota^V_d,p^W_d,p^V_d$ are the canonical embeddings and projections of $W$ and $V$, respectively. This yields that $W_d$ is a $\f_{dd}$-invariant linear subspace of $V_d$ and that $$(\f\restriction_W)_{dd}=\f_{dd}\restriction_{W_d}.$$

Now, let $\pi\colon V\to V/W$ be the canonical projection and let $\overline V=V/W$. Let $\overline V_c=\pi(V_c)$ and $\overline V_d=\pi(V_d)$; then $\overline V_c$ is linearly compact and open, while $\overline V_d$ is discrete. Since $\overline V_c$ is open in $\overline V$, we have $\overline V=\overline V_c\oplus\overline V_d$.

Clearly, there exists a canonical isomorphism $\alpha\colon \overline V_d\to V_d/W_d$ of discrete vector spaces making  the following diagram
\begin{equation*}
\xymatrix{
\overline V_d\ar@/^2pc/[rrr]^{\overline\f_{dd}} \ar[r]^{\iota^{\overline V}_d}\ar[d]^\alpha&\overline V\ar[r]^{\overline\f}&\overline V\ar[r]^{p^{\overline V}_d}&\overline V_d\ar[d]^\alpha\\
V_d/W_d \ar[rrr]^{\overline{\phi_{dd}}}&&&V_d/W_d
}
\end{equation*}
commute, where  $\overline{\f_{dd}}$ is the endomorphism induced by $\phi_{dd}$.
Then, by Propositions~\ref{prop:resdd} and \ref{conj}, $$\ent(\phi)=\ent(\phi_{dd}),\quad \ent(\f\restriction_W)=\ent(\f_{dd}\restriction_{W_d})\quad \text{and}\quad \ent(\overline{\f})=\ent(\overline{\f_{dd}}).$$
Since $\ent(\phi_{dd})=\ent(\f_{dd}\restriction_{W_d})+\ent(\overline{\f_{dd}})$, in view of the Addition Theorem for $\ent_{\dim}$ (see \cite[Theorem 5.1]{GBSalce}) and Lemma~\ref{entdim}, we can conclude that $\ent(\phi)=\ent(\phi\restriction_W)+\ent(\overline\phi)$.
\end{proof}

\end{document}